\documentclass[%
aip,
amsmath,amssymb,
reprint, onecolumn 
]{revtex4-1}

\usepackage{graphicx}
\usepackage{dcolumn}
\usepackage{bm}

\usepackage{amsthm}

\usepackage[utf8]{inputenc}
\usepackage[T1]{fontenc}
\usepackage{mathptmx}
\usepackage{etoolbox}

\usepackage{tikz}
\usetikzlibrary{calc, arrows.meta}

\usepackage{amssymb}
\usepackage{amsfonts}
\usepackage{amssymb,latexsym}
\usepackage{enumerate}
\usepackage{mathrsfs}
\usepackage{hyperref}
\usepackage{tikz}
\usetikzlibrary{calc, arrows.meta} 
\usepackage[margin=1in]{geometry}
\usepackage{color}\makeatletter
\@namedef{subjclassname@2010}{%
	\textup{2010} Mathematics Subject Classification}

\newtheorem{theorem}{Theorem}[section]

\newtheorem{proposition}[theorem]{Proposition}
\newtheorem{corollary}[theorem]{Corollary}

\newtheorem{definition}[theorem]{Definition}

\newtheorem{remark}[theorem]{Remark}

\makeatletter
\def\@email#1#2{%
	\endgroup
	\patchcmd{\titleblock@produce}
	{\frontmatter@RRAPformat}
	{\frontmatter@RRAPformat{\produce@RRAP{*#1\href{mailto:#2}{#2}}}\frontmatter@RRAPformat}
	{}{}
}%
\makeatother

\begin{document}

\title[On $p$-adic spectral zeta functions]
{On $p$-adic  spectral zeta functions}



\author{Su Hu}\altaffiliation[Email address:]{ mahusu@scut.edu.cn}
\address{Department of Mathematics, South China University of Technology, Guangzhou, Guangdong 510640, China}

\author{Min-Soo Kim}\altaffiliation[Author to whom correspondence should be addressed:]{ mskim@kyungnam.ac.kr}
\address{Department of Mathematics Education, Kyungnam University, Changwon, Gyeongnam 51767, Republic of Korea}




\begin{abstract}
The spectral zeta functions have been found many application in several branches of modern physics, including the quantum field theory, the string theory and the cosmology.
In this paper, we shall consider the spectral zeta functions and their functional determinants in the $p$-adic field.
Our  approach for the constructions of $p$-adic spectral zeta functions is to apply the $p$-adic Mellin transforms with respect to locally analytic functions $f$,
where $f$ interpolates the spectrum for the Hamiltonian of a quantum model.
\end{abstract}
 \maketitle

\section{Introduction}\label{Section 1}

Throughout this paper we shall use the following notations.
\begin{equation*}
\begin{aligned}
\qquad \mathbb{N}  ~~~&- ~\textrm{the set of positive integers}.\\
\qquad \mathbb{C}  ~~~&- ~\textrm{the field of complex numbers}.\\
\qquad p  ~~~&- ~\textrm{an odd rational prime number}. \\
\qquad\mathbb{Z}_p  ~~~&- ~\textrm{the ring of $p$-adic integers}. \\
\qquad\mathbb{Q}_p~~~&- ~\textrm{the field of fractions of}~\mathbb Z_p.\\
\qquad\mathbb C_p ~~~&- ~\textrm{the completion of a fixed algebraic closure}~\overline{\mathbb Q}_p~ \textrm{of}~\mathbb Q_{p}.
\end{aligned}
\end{equation*}

In quantum mechanics, consider a particle of mass $m$ moving in a potential $V(r)$, the time-independent state of the particle $\Psi(r)$
satisfies the Schr\"odinger equation
\begin{equation}\label{(1)}
H \Psi(r)=E \Psi(r),
\end{equation}
where $H=-\frac{{\hbar}^{2}}{2m} \Delta +V(r)$ is the Hamiltonian, $\hbar$ is the Planck constant and
$$\Delta=\frac{{\partial}^2}{\partial x^2}+\frac{{\partial}^2}{\partial y^2}+\frac{{\partial}^2}{\partial z^2}$$
is the Laplacian. The following are several known examples of (\ref{(1)}).

For  scattering of a one dimensional particle by a infinite rectangular potential barrier with the potential function
\begin{equation}
V(x)=\begin{cases} 
0, &|x|<b,\\
+\infty, &|x|\geq b,
\end{cases}
\end{equation}
the Schr\"odinger equation (\ref{(1)}) of a particle inside becomes to
\begin{equation}
\begin{cases} 
-\frac{\hbar^2}{2m}\frac{d^2}{dx^2}\Psi(x)=E\Psi(x), &|x|<b,\\
\Psi(x)=0, &|x|\geq b.
\end{cases}
\end{equation}
In this case, the solution of the above equation gives the energy levels \begin{equation}\label{barrier} E_{n}=\frac{n^{2}\pi^{2}\hbar^{2}}{8mb^2}\quad (n\in\mathbb N).\end{equation}

For a one dimensional Harmonic oscillator with the frequency $\omega$, the potential function is $$V(x)=\frac{1}{2}m\omega^{2}x^{2}$$ 
and the Schr\"odinger equation becomes to
\begin{equation}
-\frac{\hbar^2}{2m}\frac{d^2 \Psi(x)}{dx^2}+\frac{1}{2}m^{2}\omega^{2}x^{2}\Psi(x)=E\Psi(x).
\end{equation}
In this case, by solving the above equation, we obtain the energy levels
 \begin{equation}\label{oscillator}
   E_{n}=\left(n+\frac{1}{2}\right)\hbar\omega\quad (n\in\mathbb N).
 \end{equation}
 
 For the hydrogen atom  moving in a Coulomb potential 
 $$V(r)=-\frac{e^2}{r},$$
 where $-e$ is the electron charge in unrationalized electrostatic units.
 The radial Schr\"odinger equation is 
 \begin{equation}
  -\frac{\hbar^{2}}{2m_e}\frac{d^{2}\Psi(r)}{dr^2}+\left[-\frac{e^2}{r}+\frac{l(l+1)\hbar^2}{2m_{e}r^2}\right] \Psi(r)=E\Psi(r),
 \end{equation}
 where $m_{e}$ is the electron mass and $l$ is the angular quantum number. In this case, the solution of the above equation gives the energy levels
  \begin{equation}\label{hydrogen} E_{n}=-\frac{m_{e}e^4}{2\hbar^{2}n^2}\quad (n\in\mathbb N).\end{equation}

In general, solving the Schr\"odinger equation (\ref{(1)}) yields the spectrum of the Hamiltonian $H$, which we denote by $\{\lambda_n\}_{n=0}^\infty$. 
Suppose these eigenvalues satisfy
\[ 0 \leq \lambda_0 \leq \lambda_1 \leq \cdots \leq \lambda_k \leq \cdots, \quad \text{with} \quad \lambda_k \to +\infty \ \text{as} \ k\to\infty\]
after suitable normalization. The corresponding spectral zeta function is then defined as \cite[(1.3)]{Voros}:
\begin{equation}\label{(7)}
Z(s)=\sum_{n=0}^{\infty}\frac{1}{\lambda_{n}^{s}},
\end{equation}
where the sum includes all terms with $\lambda_n \neq 0$.
And the Hurwitz-type spectral zeta function is defined as  \cite[(3.1)]{Voros}:
\begin{equation}\label{(10)}
Z(s,\lambda)=\sum_{n=0}^{\infty}\frac{1}{(\lambda_{n}+\lambda)^{s}}
\end{equation}
for  $\lambda\not\in(-\infty, -\lambda_{0}]$.

These functions have been found many application in several branches of modern physics, including the quantum field theory, the string theory and the cosmology.
Their  properties   and physical applications have been investigated 
by Hawking \cite{Hawking}, Voros \cite{Voros1987, Voros}, Freitas \cite{Freitas}, and more recently by Zhang, Li, Liu and Dai \cite{Zhang},
Fedosova, Rowlett and Zhang \cite{Fedosova}, 
Reyes Bustos and Wakayama \cite{Bustos}, Kimoto and Wakayama \cite{KW}, Cunha and Freitas \cite{CF}, et.al.
As pointed out by Hawking \cite[p. 134]{Hawking}, the zeta function technique can be applied to calculate the partition functions for thermal gravitons and matter quanta on black hole and de Sitter backgrounds.
We also refer a recent book by Elizalde \cite{Elizalde}  for a complete treatments of physical applications of the spectral zeta functions.

It is seen that for the integer spectrum
\begin{equation}\label{(in)}
\lambda_{n}=n\quad (n\in\mathbb{N}_{0}),
\end{equation}
(\ref{(7)}) and (\ref{(10)}) become to
\begin{equation}\label{(7+)}
\zeta(s)=\sum_{n=1}^{\infty}\frac{1}{n^{s}}
\end{equation}
and 
\begin{equation}\label{(10+)}
\zeta(s,\lambda)=\sum_{n=0}^{\infty}\frac{1}{(n+\lambda)^{s}},
\end{equation} 
the classical Riemann zeta and Hurwitz zeta functions, respectively.

In the region of convergence, $Z(s)$ and $Z(s,\lambda)$ are given by the Mellin transforms
\begin{equation}\label{(5)}
Z(s)=\frac{1}{\Gamma(s)}\int_{0}^{\infty}\theta(t)t^{s-1}dt
\end{equation}
and 
\begin{equation}\label{(3)}
Z(s,\lambda)=\frac{1}{\Gamma(s)}\int_{0}^{\infty}\theta(t)e^{-\lambda t}t^{s-1}dt,
\end{equation}
where $$\theta(t)=\sum_{k=0}^{\infty}e^{-t\lambda_{k}}~~(t>0)$$ 
is the corresponding theta function (see \cite[(1)]{KW} and \cite[(3.3)]{Voros}).
Suppose $\theta(t)$ has an asymptotic expansion
\begin{equation}
\theta(t)\sim\sum_{m=0}^{\infty}c_{m}t^{m} \quad(t\to 0^{+}),
\end{equation}
which means that for any non-negative integer $N$,
$\theta(t)$ has the expansion 
\begin{equation}
\theta(t)=\sum_{m=0}^{N-1}c_{m}t^{m}+O(t^{N})\quad(t\to 0^{+}).
\end{equation}
Voros (see \cite[(2.12)]{Voros}) stated a  formula for the special values of $Z(s)$ at the non-positive integers $-m~(m\in\mathbb{N}_{0})$:
\begin{equation}\label{(6)}
Z(-m)=(-1)^{m}m!c_{m}.
\end{equation}
He further defined the functional determinant 
\begin{equation}\label{(11)}
\log D(\lambda)=-\frac{\partial Z(s,\lambda)}{\partial s}\bigg|_{s=0}
\end{equation}
(see \cite[(3.5)]{Voros}), which has a connection with Euler's Gamma function $\Gamma(\lambda)$
for the integer spectrum $\lambda_{n}=n~(n\in\mathbb{N}_{0}),$
that is, in this case it has an explicit expression
$$D(\lambda)=\frac{\sqrt{2\pi}}{\Gamma(\lambda)}$$
(see \cite[(3.17)]{Voros}).
It is known that the functional determinant has many physical applications, especially, it is deeply related to the integrability of the system, see \cite{Branson}. And  the following are the integral representation and the functional equation of
$\log D(\lambda)$, respectively:
\begin{equation}\label{(8)}
\log D(\lambda)=\int_{0}^{\infty}\frac{\theta(t)}{t}e^{-\lambda t}dt
\end{equation}
and 
\begin{equation}\label{(9)}
\log D(i\lambda)+\log D(-i\lambda)=-\int_{C}\frac{\theta(\tau)}{\tau}e^{i\lambda \tau}d\tau
\end{equation}
(see \cite[(3.7) and (3.14)]{Voros}).
Here  the integral path $C=L-L''$ is illustrated in the following figure (see \cite[Fig. 2]{Voros}):
\begin{figure}[ht] 
  \centering 
  \begin{tikzpicture}[>=Stealth, x=1cm, y=1cm] 

    \draw[->] (-2,0) -- (2,0) node[right]{$\mathrm{Re}\,\tau$}; 
    \draw[->] (0,-3.5) -- (0,3.5) node[above]{$\mathrm{Im}\,\tau$}; 

    \draw (1.2,0.1) -- (1.2,-0.1) node[below]{$2\pi$}; 
    \node[below left, xshift=-1pt, yshift=-1pt] at (0,0) {$0$}; 

    \draw[thick,->] (0,0) -- ++(-0.7,3) node[pos=0.93,right]{$L''$}; 
    \draw[thick,->] (0,0) -- ++(0.7,3) node[pos=0.93,left]{$L$};   
    \draw[thick,->] (0,0) -- ++(0.7,-3) node[pos=0.7,right]{$L'$}; 

    \foreach \y in {1,2,3}{
      \fill (0,\y) circle (1.5pt); 
    }
    \foreach \y in {1,2,3}{
      \fill (0,-\y) circle (1.5pt); 
    }


    \draw[thick,<-] (-0.59,1.27) .. controls (-0.65,2) 
    and (-0.8,3) .. (-1,3.4)                         
        node[pos=0, right, xshift=1mm] {};           

    \draw[thick,>-] (0.59,1.3) .. controls (0.65,2)  
    and (0.8,3) .. (1,3.4)                           
        node[pos=0, right, xshift=1mm] {$C$};          

    \draw[thick] (0.6,1.38) arc (0:-180:0.6);        

  \end{tikzpicture} 
\caption{The contour integral path} 
\end{figure} 

In the case of the integer spectrum, (\ref{(9)}) reduces to the well-known reflection formula for the Gamma function:
\begin{equation}
\Gamma(\lambda)\Gamma(1-\lambda)=\frac{\pi}{\sin \pi\lambda}.
\end{equation}

As a companion of (\ref{(10)}), we may also consider the  alternating form in parallel (see \cite[(5.34)]{Voros}): 
\begin{equation}\label{(16)}
Z^{\textrm{P}}(s,\lambda)=2\sum_{n=0}^{\infty}\frac{(-1)^{n}}{(\lambda_{n}+\lambda)^{s}}
\end{equation}
for $\lambda\not\in(-\infty, -\lambda_{0}].$
 For $\lambda_{n}=n~(n\in\mathbb{N}_{0})$, $Z^{\textrm{P}}(s,\lambda)$ reduces to the classical Hurwitz-type
Euler zeta function $\zeta_{E}(s,\lambda)$:
\begin{equation}\label{(16+)}
\zeta_{E}(s,\lambda)=2\sum_{n=0}^{\infty}\frac{(-1)^{n}}{(n+\lambda)^{s}}.
\end{equation}
Comparing the Hurwitz zeta function $\zeta(s,\lambda)$ and its alternating form $\zeta_{E}(s,\lambda)$,
the alternating form seems to have more advantages in analysis in some sense, since  $\zeta_{E}(s,\lambda)$
can be analytically continued as an analytic function in the whole complex plane,
while $\zeta(s,\lambda)$ has a simple pole at the point $s=1$. 
During the recent years, many analytic properties of $\zeta_{E}(s,\lambda)$ have been systematically studied, including the Fourier expansion, power series and asymptotic expansions, 
 integral representations, special values, and the  convexity (see \cite{Cvijovic, HKK, HK2019, HK2022, HK2024-JMAA}). 
In number theory, it has been found that  $\zeta_{E}(s,\lambda)$ can be used to represent a partial zeta function of cyclotomic fields in one version of Stark's conjectures (see \cite[p. 4249, (6.13)]{HK-G}).

It is known that  $p$-adic fields have many advantages to consider the analytic problems, especially on the convergence of the series.
Because a series $\sum_{n=1}^{\infty}a_{n}$ is convergent in the $p$-adic field if and only if $a_{n}\to 0$ as
$n\to\infty$.
In \cite{HK2021}, addressing  to a Hilbert's problem \cite{HilbertProb}, we have found an infinite order
linear differential equation satisfied by $\zeta_{p,E}(s,\lambda)$, which is convergent in certain area of the $p$-adic complex domain $\mathbb{C}_{p}$.
And in  \cite{KW}, during their investigating the quantum Rabi model (QRM), Kimoto and Wakayama obtained the divergent series expression for
$\zeta(n,\lambda)$, the special values of the Hurwitz zeta functions at positive integers (see \cite[(38) and (40)]{KW}),
but they have showed that the corresponding series for the $p$-adic Hurwitz zeta functions $\zeta_{p}(s,\lambda)$ are convergent (see \cite[(50) and Remark 7.3]{KW}).
In this paper, we shall investigate the spectral zeta functions (\ref{(10)}) and (\ref{(16)}), and their functional determinants (\ref{(11)})  in the $p$-adic field.

It needs to mention that  there are two radically different types of $p$-adic analysis. The first type considers functions from $\mathbb{Q}_{p}$ to the complex filed $\mathbb{C}$, 
while the second functions from the $p$-adic complex plane $\mathbb{C}_{p}$ to  $\mathbb{C}_{p}.$
For functions from $\mathbb Q_p$ to $\mathbb C,$ the Taibleson-Vladimirov operator $\mathbf D^\alpha,\,\alpha>0,$ is defined
as$$(\mathbf D^\alpha)\varphi(x)=\frac{1-p^{\alpha}}{1-p^{-\alpha-N}}\int_{\mathbb Q_p}|y|^{-\alpha-1}(\varphi(x-y)-\varphi(x))dy$$
on the space of locally  constant functions $\varphi(x)$ satisfying
$$\int_{|x|_p\geq1}|x|^{-\alpha-1}|\varphi(x)|dx<\infty,$$
where $dx$ denotes the normalized Haar measure of $\mathbb Q_p$ (see \cite{Kochubei2001} and \cite{VVZ}). This operator
is a good analog of the standard Laplacian, because the evolution equation
$$\frac{\partial u(x,t)}{\partial t}+\mathbf D^\alpha u(x,t)=0,\quad x\in\mathbb Q_p,\quad t\geq0,$$
describes a particle performing a random motion in $\mathbb Q_p.$ By performing a Wick
rotation $t \to it,$ with $i = \sqrt{-1},$ one obtains a free Schr\"odinger equation in
natural units is
$$i\frac{\partial \Psi(x,t)}{\partial t}=\mathbf D^\alpha \Psi(x,t),\quad x\in\mathbb Q_p,\quad t\geq0.$$
The spectra of operators of the type $\mathbf D^\alpha+V(x)$ for many potentials $V,$ have
been studied extensively (see \cite{Kochubei2001} and \cite{VVZ}), and the corresponding spectral
zeta functions have been investigated in \cite{Chacon}. This  is an up-and-coming research area which may have connections with central problems such as the distribution of primes and the Riemann hypothesis (see \cite{Shai}).

 In the present work, we shall consider the second type $p$-adic analysis which addressing the  functions from $\mathbb{C}_{p}$ to $\mathbb{C}_{p}$.
In this situation, by the first paragraph of \cite{Kochubei}, no general concept of a `Laplacian'  is known in $p$-adic analysis,
and the $p$-adic spectral theory provides no tool to establish the `Hermitian' property of an operator, other than to construct its eigenbasis (also see the work by Vishik \cite{Vishik}).
Thus for a quantum system, we can not obtain the energy levels directly by solving differential equations (\ref{(1)}) in the $p$-adic fields.
So our approach for the constructions of $p$-adic  spectral zeta functions in this situation, is to apply the $p$-adic Mellin transforms with respect to locally analytic functions $f$.
Here $f$ interpolates the spectrum for the Hamiltonian of a quantum model (see (\ref{interp})). 
 
With the above considerations, let $$D_{1}:=\{a\in\mathbb{C}_{p}: |a|_{p}\leq 1\}$$ be the unit disk of the $p$-adic complex plane $\mathbb{C}_{p}.$
First, we consider a  $p$-adic function $f(a)$ such that $f(a)~~\textrm{or}~~g(a)=\frac{1}{f(a)}$ is locally analytic on $D_{1}$, which may interpolate the spectrum for the Hamiltonian of a quantum mode.
Recall that a  $p$-adic function $h: D_{1}\to\mathbb{C}_{p}$ is said to be locally analytic, if for each $a\in D_{1}$, there is a neighborhood  $V\subset D_{1}$ of $a$ such that $h|_{V}$ is analytic
(see \cite[p. 69, Definition 25.2]{SC}). 
For example, the $p$-adic functions
\begin{equation}\label{interp} f(a)=a, \,a^{2},\, a+\frac{1}{2},\, \frac{1}{a^2} 
\end{equation}
are all satisfying our requirement. They interpolate the integer spectrum (\ref{(in)}) and the spectrums of the quantum models mentioned at 
the beginning of this paper, including the infinite rectangular potential barrier (\ref{barrier}), the harmonic oscillator (\ref{oscillator}) and the hydrogen atom (\ref{hydrogen}), respectively, except for some constants.

Recall that the $p$-adic analog of the classical  Hurwitz zeta function $\zeta(s,\lambda)$ (see (\ref{(10+)})) is defined by the $p$-adic Mellin 
transform of the Haar distribution:
\begin{equation}\label{p-adic Hurwitz+}
\zeta_{p}(s,\lambda)=\frac{1}{s-1}\int_{\mathbb{Z}_{p}}\langle \lambda+a \rangle^{1-s} da
\end{equation}
for $\lambda\in\mathbb{C}_{p}\backslash \mathbb{Z}_{p}$
(see \cite[p. 283, Definition 11.2.5]{Cohen}).
On the other side, consider the $p$-adic measure $\mu_{-1}$ defined by
\begin{equation}\label{minus}
\mu_{-1}\left(a + p^{N}\mathbb{Z}_{p}\right) = (-1)^{a}
\end{equation}
for $0 \leq a < p^{N}$. This measure was independently found by Katz \cite[p.~486]{Katz} (in Katz's notation, the $\mu^{(2)}$-measure), Shiratani and
Yamamoto \cite{Shi}, Osipov \cite{Osipov}, Lang~\cite{Lang} (in Lang's notation, the $E_{1,2}$-measure), T. Kim~\cite{TK} from very different viewpoints.
Obviously, in contrast with the Haar distribution, the $\mu_{-1}$-measure is bounded under the $p$-adic valuation, so it can be applied to integrate the continuous functions on $\mathbb{Z}_{p}$
(see \cite[p. 39, Theorem 6]{Ko}).
The $p$-adic analog of the classical Hurwitz-type
Euler zeta function $\zeta_{E}(s,\lambda)$ (see (\ref{(16+)}))
is then defined by the $p$-adic Mellin transform of the $\mu_{-1}$-measure:
\begin{equation}\label{p-adic Euler}
\begin{aligned}
\zeta_{p,E}(s,\lambda)&=\int_{\mathbb{Z}_{p}}\langle \lambda+a \rangle^{1-s} d\mu_{-1}(a)\\
&=\lim_{N\to\infty}\sum_{a=0}^{p^{N}-1}\langle \lambda+a\rangle^{1-s}(-1)^{a}
\end{aligned}
\end{equation}
for $\lambda\in\mathbb{C}_{p}\backslash \mathbb{Z}_{p}$ (see \cite[p. 2985, Definition 3.3]{HK2012}).

In the following, in analogy with the definitions of the spectral zeta functions in the complex case (see (\ref{(10)}) and (\ref{(16)})), we shall generalize (\ref{p-adic Hurwitz+}) and (\ref{p-adic Euler})  
by applying the $p$-adic Mellin transforms to  the $p$-adic locally analytic functions $f$.

\begin{definition}\label{zeta}
Denote 
$$\mathscr{F}:=\{f(a): a\in\mathbb{Z}_{p}\}$$ 
by the value set of $f$ on $\mathbb{Z}_{p}.$ For $\lambda\in\mathbb{C}_{p}$ such that $-\lambda\not\in\mathscr{F}$, we define the $p$-adic counterparts of (\ref{(10)}) and (\ref{(16)})
by the Mellin transforms:
\begin{equation}\label{(Def1)}
\zeta_{p}^{f}(s,\lambda)=\frac{1}{s-1}\int_{\mathbb{Z}_{p}}\langle \lambda+f(a) \rangle^{1-s} da
\end{equation}
and 
\begin{equation}\label{(Def2)}
\zeta_{p,E}^{f}(s,\lambda)=\int_{\mathbb{Z}_{p}}\langle \lambda+f(a) \rangle^{1-s} d\mu_{-1}(a),
\end{equation}
respectively.
And we name them the $p$-adic Hurwitz zeta function and the $p$-adic Hurwitz-type Euler zeta function with respect to $f$, respectively.
\end{definition}
\begin{remark} Obviously, for $f(a)=a$, which interpolates the integer spectrum, $\zeta_{p}^{f}(s,\lambda)$ and $\zeta_{p,E}^{f}(s,\lambda)$ reduce to
the classical $p$-adic Hurwitz zeta function $\zeta_{p}(s,\lambda)$ (see (\ref{p-adic Hurwitz})) and the classical $p$-adic Hurwitz-type Euler zeta function $\zeta_{p,E}(s,\lambda)$ (see (\ref{p-adic Euler})), respectively.
\end{remark}
\begin{remark}
In the case of $g(a)=\frac{1}{f(a)}$ is locally analytic on $D_{1}$, instead of $f(a)$, we may consider the $p$-adic
Hurwitz zeta function for $g(a)$ in the above definition. Because in the following, we will immediately show that both 
$\zeta_{p}^{g}(s,\lambda)$ and $\zeta_{p,E}^{g}(s,\lambda)$  can be analytically continued as a $C^{\infty}$-function for $s$ on $\mathbb{Z}_{p}\backslash\{1\}$ and on $\mathbb{Z}_{p}$, respectively (see Theorem \ref{Theorem 3.4}).
\end{remark}
\begin{remark} For (\ref{(Def2)}), the definition of  the $p$-adic Hurwitz-type Euler zeta function with respect to $f$, we need only to assume $f$ is continuous on $\mathbb Z_p$, since the $\mu_{-1}$-measure is bounded under the $p$-adic valuation (see \cite[p. 39, Theorem 6]{Ko}).
\end{remark}
The remaining parts of this paper will be organized as follows. In Section \ref{Section 2}, we briefly recall the use of Haar distribution and $\mu_{-1}$-measure in defining $p$-adic zeta functions. In Section \ref{Section 3}, we devote to the study of the properties for the $p$-adic functions $\zeta_{p}^{f}(s,\lambda)$ and $\zeta_{p,E}^{f}(s,\lambda)$
in a parallel way. We first consider the definition areas and the analyticities of  $\zeta_{p}^{f}(s,\lambda)$ and $\zeta_{p,E}^{f}(s,\lambda)$ for the variables $(s,\lambda)$ (see Theorem \ref{Theorem 3.4}). Then we prove their
fundamental properties, including the convergent power series expansions and the derivative formulas (see Theorems \ref{Theorem 3.6} and \ref{Theorem 3.10}). Furthermore, the convergent power 
series expansions imply formulas on their special values at integers (see Remarks \ref{Remark} and \ref{3.5}).
In Section \ref{Section 4}, in analogy with the definition of the functional determinant (see (\ref{(11)})), we define the $p$-adic log Gamma functions (the $p$-adic functional determinants) $\log\Gamma_{p}^{f}(\lambda)$ and $\log\Gamma_{p,E}^{f}(\lambda)$ as the derivatives of $\zeta_{p}^{f}(s,\lambda)$ and $\zeta_{p,E}^{f}(s,\lambda)$
 at $s=0$, respectively (see Definition \ref{Definition 4.1}). We will show the integral representations and Stirling's series for them (see Theorems \ref{Theorem 4.2} and \ref{Theorem 4.3}). 
In Section \ref{Section 5}, following the referee's suggestion, we present two constructive algorithms  to determine the $p$-adic analytic  function $f$ on $D_{1}$ which  interpolates the spectra of a quantum model. These algorithms are fundamentally based on the methods developed in Chapter 3 of Iwasawa's treatise \cite{Iw}.
 
  \section{The $p$-adic Hurwitz zeta functions}\label{Section 2}
 \subsection{$p$-adic Teichmüller character and projection function}
To introduce the definitions and the properties of $p$-adic Hurwitz zeta functions, we need to recall some concepts in $p$-adic analysis, which focus on the $p$-adic Teichmüller character $\omega_\nu(a)$ and the projection function $\langle a \rangle$. Our exposition aligns with the framework established in \cite{TP}.

For $a \in \mathbb{Z}_p$ with $p \nmid a$, there exists a unique $(p-1)$th root of unity $\omega(a) \in \mathbb{Z}_p$ satisfying
\[
a \equiv \omega(a) \mod p,
\]
where $\omega$ denotes the Teichmüller character. The projection function $\langle a \rangle$ is then defined as
\[
\langle a \rangle = \omega^{-1}(a) a,
\]
ensuring $\langle a \rangle \equiv 1 \mod p$.

We may extend the definition of  $\langle a \rangle$  from $\mathbb{Z}_{p}$ to $\mathbb{C}_{p}$ as follows.
For general $a \in \mathbb{C}_p^\times$, fix an embedding of $\overline{\mathbb{Q}}$ into $\mathbb{C}_p$. Let $a = p^{\nu_p(a)} u$ with $|u|_p = 1$. The Teichmüller representative $\hat{a}$ is uniquely determined by
\[
|\hat{a} - a|_{p} < 1 \quad \text{and} \quad \hat{a} = \lim_{n \to \infty} a^{p^n}.
\]
And the projection function extends to $\mathbb{C}_p^\times$ via
\[
\langle a \rangle = p^{-\nu_p(a)} \cdot \frac{a}{\hat{a}}.
\]
This induces the decomposition:
\[
\mathbb{C}_p^\times \simeq p^{\mathbb{Q}} \times \mu \times D,
\]
where $\mu$ consists of roots of unity with order coprime to $p$ and $D = \{ x \in \mathbb{C}_p : |x - 1|_p < 1 \}$.
Although the decomposition of $\mathbb{C}_p^\times$ depends on the choice of embedding of $\mathbb{Q}$ into $\mathbb{C}_p$; for fixed $a \in \mathbb{C}_p^\times$, the components $p^{\nu_p(a)}$, $\hat{a}$, and $\langle a \rangle$ are uniquely determined up to roots of unity. 

Now define  the  $p$-adic Teichmüller character $\omega_\nu(a)$ on $\mathbb{C}_p^\times$ by
\[
\omega_\nu(a) = \frac{a}{\langle a \rangle} = p^{\nu_p(a)} \cdot \hat{a}.
\]
We see that the projections $a \mapsto p^{\nu_p(a)}$ and $a \mapsto \hat{a}$ are locally constant, hence have zero derivative,
and the projection $a \mapsto \langle a \rangle$ satisfies
\begin{equation}\label{deri}
   \frac{d}{da}\langle a\rangle=\frac{\langle a\rangle}{a}.
\end{equation}
 
 \subsection{$p$-adic Hurwitz zeta functions}
 In this section, we briefly recall the use of Haar distribution and $\mu_{-1}$-measure in defining $p$-adic zeta functions.

It is known that the  $p$-adic analog of the classical  Hurwitz zeta function $\zeta(s,\lambda)$ (see (\ref{(10+)})) is defined by the $p$-adic Mellin 
transform of the Haar distribution:
\begin{equation}\label{p-adic Hurwitz}
\begin{aligned}
\zeta_{p}(s,\lambda)&=\frac{1}{s-1}\int_{\mathbb{Z}_{p}}\langle \lambda+a \rangle^{1-s} da\\
&=\frac{1}{s-1}\lim_{N\to\infty}\frac{1}{p^{N}}\sum_{a=0}^{p^{N}-1}\langle \lambda+a\rangle^{1-s}
\end{aligned}
\end{equation}
for $\lambda\in\mathbb{C}_{p}\backslash \mathbb{Z}_{p}$
(see \cite[p. 283, Definition 11.2.5]{Cohen}). And it interpolates (\ref{(10+)}) at non-positive integers, that is, for $m\in\mathbb{N}$ we have
\begin{equation}
\begin{aligned}
\zeta_{p}(1-m,\lambda)=-\frac{1}{\omega_{v}^{m}(\lambda)}\frac{B_{m}(\lambda)}{m}
=\frac{1}{\omega_{v}^{m}(\lambda)}\zeta(1-m,\lambda),
\end{aligned}
\end{equation}
where $B_{m}(\lambda)$ is the $m$th Bernoulli polynomial defined by the generating function
\begin{equation}
\frac{te^{\lambda t}}{e^{t}-1}=\sum_{m=0}^{\infty}B_{m}(\lambda)\frac{t^{m}}{m!}
\end{equation}
(see \cite[p. 284, Proposition 11.2.6]{Cohen}). The distribution corresponding to the integral (\ref{p-adic Hurwitz})
is defined by
\begin{equation}\label{Haar}
\mu_{\textrm{Haar}}(a+p^{N}\mathbb{Z}_{p})=\frac{1}{p^{N}}
\end{equation}
for $a\in\mathbb{Z}_{p},$
which is named the Haar distribution.
 Since 
\begin{equation}
\begin{aligned}
\lim_{N\to\infty} |\mu_{\textrm{Haar}}(a+p^{N}\mathbb{Z}_{p})|_{p}&=\lim_{N\to\infty}\left|\frac{1}{p^{N}}\right|_{p}\\
&=\lim_{N\to\infty}p^{N}=\infty,
\end{aligned}
\end{equation}
it is an unbounded $p$-adic distribution, and it can be applied 
 to integrate the $C^{1}$-functions (the continuously differentiable functions) on $\mathbb Z_p$ (see \cite[p. 167, Definition 55.1]{SC}). This integral is named the Volkenborn integral in many literatures.

On the other side, the $p$-adic Hurwitz-type Euler zeta function $\zeta_{p,E}(s,\lambda)$, defined in (\ref{p-adic Euler}), interpolates  $\zeta_{E}(s,\lambda)$ at non-positive integers. That is, for $m\in\mathbb{N}$ we have
\begin{equation}  
\zeta_{p,E}(1-m,\lambda) = \frac{1}{\omega_{v}^{m}(\lambda)} E_{m}(\lambda) = \frac{2}{\omega_{v}^{m}(\lambda)} \zeta_{E}(-m,\lambda),  
\end{equation}  
where $E_{m}(\lambda)$ denotes the $m$-th Euler polynomial, defined by the generating function  
\begin{equation}  
\frac{2e^{\lambda t}}{e^{t} + 1} = \sum_{m=0}^{\infty} E_{m}(\lambda) \frac{t^{m}}{m!}  
\end{equation}  
(see \cite[p. 2986, Theorem 3.8(2)]{HK2012}).  
From the corresponding properties of the $\mu_{-1}$-measure (see \cite[p. 2982, Theorem 2.2]{HK2012}), in \cite{HK2012} we have proved many fundamental 
properties for $\zeta_{p,E}(s,\lambda)$, including the convergent Laurent series
expansion, the distribution formula, the functional equation, the
reflection formula, the derivative formula and the $p$-adic Raabe
formula. Using these zeta function as building blocks, we defined  the corresponding 
$p$-adic $L$-functions $L_{p,E}(\chi,s)$, which has been connected with the arithmetic theory of cyclotomic fields in algebraic number theory.
In concrete, the Kubota-Leopoldt's $p$-adic $L$-function
\begin{equation}
L_{p}(s,\chi)=\frac{1}{s-1}\int_{\mathbb{Z}_{p}}\chi(a)\langle a \rangle^{1-s} da
\end{equation}
is corresponding to the ideal class group of the $p^{n}$th cyclotomic field $\mathbb{Q}(\zeta_{p^n})$ (see \cite{Iw}), while
the $p$-adic $L$-function 
\begin{equation}
L_{p,E}(s,\chi)=\int_{\mathbb{Z}_{p}}\chi(a)\langle a \rangle^{1-s} d\mu_{-1}(a)
\end{equation}
is corresponding to the $(S,\{2\})$-refined ideal class group of $\mathbb{Q}(\zeta_{p^n})$.
For details, we refer the readers to \cite{HK2016}. Recently, address to a Hilbert's problem \cite{HilbertProb}, in \cite{HK2021} we also found an infinite order
linear differential equation satisfied by $\zeta_{p,E}(s,\lambda)$, which is convergent in certain area of the $p$-adic complex domain $\mathbb{C}_{p}$.

\section{The properties of  $\zeta_{p}^{f}(s,\lambda)$ and $\zeta_{p,E}^{f}(s,\lambda)$}\label{Section 3}
Generalizing $\zeta_{p}(s,\lambda)$  and   $\zeta_{p,E}(s,\lambda)$, we have defined the  $p$-adic Hurwitz zeta function $\zeta_{p}^{f}(s,\lambda)$  
and the $p$-adic Hurwitz-type Euler zeta function $\zeta_{p,E}^{f}(s,\lambda)$ with respect to locally analytic functions $f$ (see Definition \ref{zeta}). 
In this section, we shall study their properties in a parallel way.

First, we refer the readers to \cite[Section 11.2.1]{Cohen}, \cite[Section 2.1]{HK2021}, \cite[p. 2984]{HK2012} and \cite[p. 1244]{TP} for the definitions of the $p$-adic Teichm\"uller character $\omega_v(\lambda)$ and
the projection function $\langle \lambda \rangle$ for $\lambda\in\mathbb{C}_{p}^{\times}=\mathbb{C}_{p}\backslash\{0\}$. For $\lambda\in\mathbb{C}_{p}^{\times}$ and $s\in\mathbb{C}_{p},$ the two-variable function
$\langle \lambda \rangle^{s}$ (see \cite[p. 141]{SC}) is defined by
\begin{equation}\label{(3.6)}
\langle \lambda\rangle^{s}=\sum_{n=0}^{\infty}\binom{s}{n}(\langle \lambda \rangle-1)^{n},
\end{equation}
when this sum is convergent. As in the classical situations \cite{HK2012} and \cite{TP}, the definition areas and the analyticities properties of $\zeta_{p}^{f}(s,\lambda)$ and $\zeta_{p,E}^{f}(s,\lambda)$
 are granted by the following properties of the $p$-adic function $\langle \lambda\rangle^{s}$.
\begin{proposition}[See Tangedal and Young \cite{TP}]\label{Proposition 3.2}
For any
 $\lambda\in\mathbb{C}_{p}^{\times}$ the function $s\mapsto\langle
 \lambda\rangle^{s}$ is a $C^{\infty}$ function of $s$ on
 $\mathbb{Z}_{p}$ and is analytic on a disc of positive radius about
 $s=0$; on this disc it is locally analytic as a function of $\lambda$ and
 independent of the choice made to define the $\langle
 \cdot\rangle$ function. If $\lambda$ lies in a finite extension $K$ of
 $\mathbb{Q}_{p}$ whose ramification index over $\mathbb{Q}_{p}$ is
 less than $p-1$ then  $s\mapsto\langle
 \lambda\rangle^{s}$  is analytic for $|s|_{p} <
 |\pi|_{p}^{-1}p^{-1/(p-1)}$, where $(\pi)$ is the maximal ideal of
 the ring of integers $O_{K}$ of $K$. If $s\in\mathbb{Z}_{p},$
 the function $\lambda\mapsto\langle \lambda\rangle^{s}$ is an analytic function of $\lambda$
 on any disc of the form $\{\lambda\in\mathbb{C}_{p}:|\lambda-y|_{p} <|y|_{p}\}$.
\end{proposition}
From the above properties, the definitions (\ref{(Def1)}) and  (\ref{(Def2)}), notice that the $p$-adic function $f(a)$ is continuous and satisfies
$|f(a)|_{p}\leq M$ for $a\in\mathbb{Z}_{p}$, we may obtain the following analyticities of the $p$-adic zeta functions  $\zeta_{p}^{f}(s,\lambda)$ and $\zeta_{p,E}^{f}(s,\lambda)$.
 \begin{theorem}[Analyticity]\label{Theorem 3.4}
 For $\lambda\in\mathbb{C}_{p}$ such that $-\lambda\not\in\mathscr{F}$,
 $\zeta_{p}^{f}(s,\lambda)$ is a $C^{\infty}$-function of $s$ on  $\mathbb{Z}_{p}\backslash\{1\}$, while $\zeta_{p,E}^{f}(s,\lambda)$ 
 is a $C^{\infty}$ function of $s$ on $\mathbb{Z}_{p}$. And they are analytic functions of $s$ on a disc of
positive radius about $s=0$; on this disc they are locally analytic as
 functions of $\lambda$ and independent of the choice made to define the
$\langle\cdot\rangle$ function. If $\lambda$ is so chosen such that for $a\in\mathbb{Z}_{p}$, $\lambda+f(a)$  lie in a
finite extension $K$ of $\mathbb{Q}_{p}$ whose ramification index
over $\mathbb{Q}_{p}$ is less than $p-1$, then $\zeta_{p}^{f}(s,\lambda)$ is
analytic for $|s|_{p} <
 |\pi|_{p}^{-1}p^{-1/(p-1)}$ except for a simple pole at $s=1$, while  $\zeta_{p,E}^{f}(s,\lambda)$ is
analytic for $|s|_{p} <
 |\pi|_{p}^{-1}p^{-1/(p-1)}$. If $s\in\mathbb{Z}_{p}\backslash\{1\}$, the function
 $\zeta_{p}^{f}(s,\lambda)$ is locally analytic as a function of $\lambda\in\mathbb{C}_{p}$ such that $-\lambda\not\in\mathscr{F}$,
 and if $s\in\mathbb{Z}_{p}$, the function
 $\zeta_{p,E}^{f}(s,\lambda)$ is locally analytic as a function of $\lambda\in\mathbb{C}_{p}$ such that $-\lambda\not\in\mathscr{F}$.
 \end{theorem}
 \begin{proof}
Fixed $\lambda\in\mathbb{C}_{p}$ such that $-\lambda\not\in\mathscr{F}$, we have $\lambda+f(a)\in\mathbb{C}_{p}^{\times}$  for any $a\in\mathbb{Z}_{p}.$ Then by  \cite[p. 1245, (2.22)]{TP},  
 \begin{equation}\label{(P+)}
 \begin{aligned}
 \langle \lambda+f(a) \rangle^{s}&=\exp_{p}(s\log_{p}\langle \lambda+f(a) \rangle) \\
& =\sum_{n=0}^{\infty}\frac{\left(s\log_{p}\langle \lambda+f(a) \rangle\right)^{n}}{n!}
   \end{aligned}
 \end{equation}
 for any $a\in\mathbb Z_p.$ Here we recall that the $p$-adic exponential function $\exp_{p}$ is defined by the power series
 $$\exp_{p}(\lambda)=\sum_{n=0}^{\infty}\frac{\lambda^{n}}{n!},$$
 which is convergent for $|\lambda|_{p}<p^{-1/(p-1)}$ (see \cite[p. 70, Theorem 25.6]{SC}).
So by (\ref{(P+)}) we have $\langle \lambda+f(a) \rangle^{s}$ is analytic of $s$ for $$|s|_{p}<\rho_{a}:=p^{-1/(p-1)}|\log_{p}\langle \lambda+f(a) \rangle|_{p}^{-1}.$$
 Since $\mathbb{Z}_{p}$ is compact, there exists a constant $\rho>0$ which does not depend on $a\in\mathbb{Z}_{p},$ such that 
 $ \langle \lambda+f(a) \rangle^{s} $ is  analytic for $|s|_{p}<\rho.$ For such $s$, since $f(a)$ is locally analytic on the disc $D_{1}\subset\mathbb{C}_{p}$, by \cite[p. 124, Theorem 42.4]{SC}, the theorem on the analyticity of the composite 
of two $p$-adic analytic functions, we see that $\langle \lambda+f(a) \rangle^{s}$ is locally analytic for $a\in D_{1}.$ Thus according to
 \cite[p. 91, Corollary 29.11]{SC} and \cite[p. 167, Definition 55.1]{SC}, by (\ref{(Def1)}) and  (\ref{(Def2)}) we conclude that $\zeta_{p}^{f}(s,\lambda)$
 and $\zeta_{p,E}^{f}(s,\lambda)$ are well-defined and analytic for $|s|_{p}<\rho.$
 
Furthermore, if   $\lambda$ is so chosen such that for $a\in\mathbb{Z}_{p}$, $\lambda+f(a)$  lie in a
finite extension $K$ of $\mathbb{Q}_{p}$ whose ramification index $e$
over $\mathbb{Q}_{p}$ is less than $p-1,$ 
then  for $a\in\mathbb{Z}_{p}$, we have 
$$\langle \lambda+f(a) \rangle-1\in (\pi)$$ and $$\left|\langle \lambda+f(a) \rangle-1\right|_{p}\leq|\pi|_{p}=p^{-\frac{1}{e}}< p^{-1/(p-1)}.$$
Then by \cite[p. 51, Lemma 5.5]{Wa}, we get
$$\left|\log_{p}\langle \lambda+f(a) \rangle\right|_{p}=\left|\langle \lambda+f(a) \rangle-1\right|_{p}\leq|\pi|_{p}$$
and $$|\log_{p}\langle \lambda+f(a) \rangle|_{p}^{-1}p^{-1/(p-1)}\geq|\pi|_{p}^{-1}p^{-1/(p-1)}.$$
So applying (\ref{(P+)}) again, we see that for any $a\in\mathbb Z_p$, $ \langle \lambda+f(a) \rangle^{s} $ is analytic for $|s|_{p}<|\pi|_{p}^{-1}p^{-1/(p-1)}.$  
By (\ref{(Def1)}) and  (\ref{(Def2)}),  
we conclude that $\zeta_{p}^{f}(s,\lambda)$ is
analytic for  $|s|_{p} < |\pi|_{p}^{-1}p^{-1/(p-1)}$ except for a simple pole at $s=1$, while  $\zeta_{p,E}^{f}(s,\lambda)$ is
analytic for $|s|_{p} <
 |\pi|_{p}^{-1}p^{-1/(p-1)}$.   \end{proof}
For $\lambda\in\mathbb{C}_{p}$, it is known that the $m$th Bernoulli  and Euler polynomials have the following $p$-adic representations  by the Haar distribution and the $\mu_{-1}$-measure, respectively:
 \begin{equation}
 B_{m}(\lambda)=\int_{\mathbb{Z}_{p}}(\lambda+a)^{m}da
 \end{equation}
 and 
  \begin{equation}
 E_{m}(\lambda)=\int_{\mathbb{Z}_{p}}(\lambda+a)^{m}d\mu_{-1}(a)
 \end{equation}
 (see \cite[p. 279, Lemma 11.1.7]{Cohen} and \cite[p. 2980, (2.6)]{HK2012}).
So we name 
 \begin{equation}\label{(B9)}
 B_{m}^{f}(\lambda)=\int_{\mathbb{Z}_{p}}(\lambda+f(a))^{m}da
 \end{equation}
 and 
  \begin{equation}\label{(E10)}
 E_{m}^{f}(\lambda)=\int_{\mathbb{Z}_{p}}(\lambda+f(a))^{m}d\mu_{-1}(a),
 \end{equation}
the $m$th Bernoulli and Euler polynomials associated with $f$, respectively.
Then by (\ref{(3.6)}) and Proposition \ref{Proposition 3.2}, we also have the following power series expansions of $\zeta_{p}^{f}(s,\lambda)$ and $\zeta_{p,E}^{f}(s,\lambda)$.
Recall that $f$ is locally analytic on $D_{1}$ thus on $\mathbb Z_p$, by \cite[p. 168, Proposition 55.4]{SC}, it is Volkenborn integrable, so  (\ref{(B9)}) is well-defined.

If $f$ is locally analytic on $D_{1}$ (thus it is locally analytic and continuous on $\mathbb{Z}_{p}$), 
then by the compactness of $\mathbb{Z}_{p}$, $f$ is bounded under the $p$-adic valuation, that is, there exists a constant $M>0$ such that
\begin{equation}\label{(17)}
|f(a)|_{p}\leq M
\end{equation}
for every $a\in\mathbb{Z}_{p}$.  
\begin{theorem}[Power series expansions]\label{Theorem 3.6}
Suppose $f$ satisfies the condition (\ref{(17)}). Then for $\lambda\in \mathbb{C}_{p}$ such that $|\lambda|_{p} > M,$ there
are identities of analytic functions
\begin{equation}\label{(P1)}
\begin{aligned}
\zeta_{p}^{f}(s,\lambda)&=\frac{1}{s-1}\int_{\mathbb{Z}_{p}}\langle \lambda+f(a) \rangle^{1-s} da\\
&=\frac{1}{s-1}\langle \lambda \rangle^{1-s}\sum_{m=0}^\infty\binom{1-s}m B_{m}^{f}(0)\frac{1}{\lambda^m}
\end{aligned}
\end{equation}
and 
\begin{equation}\label{(P2)}
\begin{aligned}
\zeta_{p,E}^{f}(s,\lambda)&=\int_{\mathbb{Z}_{p}}\langle \lambda+f(a) \rangle^{1-s} d\mu_{-1}(a)\\
&=\langle \lambda \rangle^{1-s}\sum_{m=0}^\infty\binom{1-s}m E_{m}^{f}(0)\frac{1}{\lambda^m}
\end{aligned}
\end{equation}
on a disc of positive radius about $s=0.$
If in addition  $\lambda$ is so chosen such that $\lambda$ and $\lambda+f(a)~(a\in\mathbb{Z}_{p})$  lie in a
finite extension $K$ of $\mathbb{Q}_{p}$ whose ramification index
over $\mathbb{Q}_{p}$ is less than $p-1$, and let $(\pi)$ be the maximal ideal of
 the ring of integers $O_{K}$ of $K$, then (\ref{(P1)}) is valid for $s$ in
$\mathbb{C}_{p}$ such that $|s|_{p} <
 |\pi|_{p}^{-1}p^{-1/(p-1)}$ except for $s=1$, while (\ref{(P2)}) is valid for $s$ in
$\mathbb{C}_{p}$ such that $|s|_{p} <
 |\pi|_{p}^{-1}p^{-1/(p-1)}$.
\end{theorem}
\begin{remark}[Special values]\label{Remark}
In the above theorem, if $\lambda$ is chosen such that $\lambda$ and $\lambda+f(a)~(a\in\mathbb{Z}_{p})$  lie in  
  $\mathbb{Q}_{p},$ then (\ref{(P1)}) is valid for $s$ in
$\mathbb{C}_{p}$ such that $|s|_{p} <p^{(p-2)/(p-1)}$ except for $s=1$, while (\ref{(P2)}) is valid for $s$ in
$\mathbb{C}_{p}$ such that $|s|_{p} <p^{(p-2)/(p-1)}$.
So by setting $s=n~(n\in\mathbb{N}\setminus\{1\})$ and $s=n~(n\in\mathbb{N})$ in (\ref{(P1)}) and (\ref{(P2)}) respectively, we have the following 
convergent power series expansions for the special values at positive integers:
\begin{equation}\label{(P1+)}
\zeta_{p}^{f}(n,\lambda) 
 =\frac{\omega_{v}^{n-1}(\lambda)}{n-1} \sum_{m=0}^\infty\binom{1-n}m B_{m}^{f}(0)\frac{1}{\lambda^{m+n-1}}
\end{equation}
and 
\begin{equation}\label{(P2+)}
\zeta_{p,E}^{f}(n,\lambda)=\omega_{v}^{n-1}(\lambda)\sum_{m=0}^\infty\binom{1-n}m E_{m}^{f}(0)\frac{1}{\lambda^{m+n-1}}.
\end{equation}
Since by  (\ref{(B9)})
 \begin{equation}\label{(B9+)}
 \begin{aligned}
 B_{n}^{f}(\lambda)&=\int_{\mathbb{Z}_{p}}(\lambda+f(a))^{n}da\\
 &=\int_{\mathbb{Z}_{p}}\sum_{m=0}^{n}\binom{n}m\lambda^{n-m}f^{m}(a)da\\
 &=\sum_{m=0}^{n}\lambda^{n-m}\int_{\mathbb{Z}_{p}}f^{m}(a)da\\
&=\sum_{m=0}^{n}\binom{n}mB_{m}^{f}(0)\frac{1}{\lambda^{m-n}},
  \end{aligned}
  \end{equation}
setting $s=1-n~(n\in\mathbb{N})$ in (\ref{(P1)}) and (\ref{(P2)}) respectively, we get the following formulas on their special values at non-positive integers: 
\begin{equation}\label{(P1++)}
\begin{aligned}
 \zeta_{p}^{f}(1-n,\lambda)&=-\frac{1}{\omega_{v}^{n}(\lambda)}\frac{1}{n}\sum_{m=0}^{n}\binom{n}mB_{m}^{f}(0)\frac{1}{\lambda^{m-n}}\\
 &=-\frac{1}{\omega_{v}^{n}(\lambda)}\frac{B_n^f(\lambda)}{n}
 \end{aligned}
  \end{equation}
and 
\begin{equation}\label{(P2++)}
\begin{aligned}
\zeta_{p,E}^{f}(1-n,\lambda)&=\frac{1}{\omega_{v}^{n}(\lambda)}\sum_{m=0}^{n}\binom{n}m E_{m}^{f}(0)\frac{1}{\lambda^{m-n}}\\
&=\frac{1}{\omega_{v}^{n}(\lambda)}E_n^f(\lambda).
\end{aligned}
\end{equation}
\end{remark}
\begin{remark}\label{3.5}
In \cite{KW}, Kimoto and Wakayama showed the following formal power series expression for the classical Hurwitz zeta functions $\zeta(s,\lambda)$ at $s=n~(n\in\mathbb{N}$)
\begin{equation}
\zeta(n,\lambda)=\sum_{m=0}^{\infty}(-1)^{m}\frac{B_{m}}{m!}\frac{(m+n-2)!}{(n-1)!}\frac{1}{\lambda^{m+n-1}}
\end{equation}
(see \cite[(38)]{KW})
and the following formal power series expansion for the Hurwitz-type spectral zeta function of the quantum Rabi model (QRM) at $s=n~(n\in\mathbb{N}$)
\begin{equation}
\zeta_{\textrm{QRM}}(n,\lambda)=\sum_{m=0}^{\infty}(-1)^{m}\frac{(RB)_{m}}{m!}\frac{(m+n-2)!}{(n-1)!}\frac{1}{\lambda^{m+n-1}},
\end{equation}
where $\textrm{(RB)}_{m}$  denotes the $m$th Rabi-Bernoulli numbers (see \cite[(40)]{KW}).  
\end{remark}
\begin{proof}[Proof of Theorem \ref{Theorem 3.6}]
For $\lambda$ and $\lambda+f(a)$ in $\mathbb{C}_{p}^{\times}$, from the multiplicative of the projection function $\langle \lambda \rangle$,
we have 
\begin{equation}
\langle \lambda+f(a) \rangle^{1-s}=\langle \lambda \rangle^{1-s}\left\langle 1+\frac{f(a)}{\lambda}\right\rangle^{1-s}.
\end{equation}
Chosen $s$ and $\lambda$ which satisfy the  assumptions of this theorem, 
since $|\lambda|_{p}>M$ and $|f(a)|_{p}\leq M$ for $a\in\mathbb{Z}_{p}$,
we have $|f(a)/\lambda|_{p}<1$ and $$\left\langle 1+\frac{f(a)}{\lambda}\right\rangle=1+\frac{f(a)}{\lambda},$$
and from the binomial theorem we get 
\begin{equation}\label{(3.7)}
\begin{aligned}
\langle \lambda+f(a) \rangle^{1-s}&=\langle \lambda \rangle^{1-s}\left(1+\frac{f(a)}{\lambda}\right)^{1-s}\\
&=\langle \lambda \rangle^{1-s}\sum_{m=0}^\infty\binom{1-s}m f^{m}(a) \frac{1}{\lambda^m}.
\end{aligned}
\end{equation}
From our assumption, $f(a)$ is locally analytic on the disc $D_{1}\subset\mathbb{C}_{p}$, by \cite[p. 124, Theorem 42.4]{SC} we have  (\ref{(3.7)}) is an identity of locally analytic functions at each $a\in\mathbb{Z}_{p}.$
Then applying \cite[p. 168, Proposition 55.2]{SC} to integrate the right hand side of (\ref{(3.7)}) with respect to $a$ term by term, from (\ref{(Def1)}) and (\ref{(B9)}), we obtain
\begin{equation} 
\begin{aligned}
\zeta_{p}^{f}(s,\lambda)&=\frac{1}{s-1}\int_{\mathbb{Z}_{p}}\langle \lambda+f(a) \rangle^{1-s} da\\
&=\frac{1}{s-1}\langle \lambda \rangle^{1-s}\sum_{m=0}^\infty\binom{1-s}m \int_{\mathbb{Z}_{p}}f^{m}(a)da \frac{1}{\lambda^m}\\
&=\frac{1}{s-1}\langle \lambda \rangle^{1-s}\sum_{m=0}^\infty\binom{1-s}m B_{m}^{f}(0)\frac{1}{\lambda^m}.
\end{aligned}
\end{equation}
Similarly, by (\ref{(Def2)}) and (\ref{(E10)}), we have
\begin{equation}
\begin{aligned}
\zeta_{p,E}^{f}(s,\lambda)&=\int_{\mathbb{Z}_{p}}\langle \lambda+f(a) \rangle^{1-s} d\mu_{-1}(a)\\
&=\langle \lambda \rangle^{1-s}\sum_{m=0}^\infty\binom{1-s}m \int_{\mathbb{Z}_{p}}f^{m}(a)d\mu_{-1}(a) \frac{1}{\lambda^m}\\
&=\langle \lambda \rangle^{1-s}\sum_{m=0}^\infty\binom{1-s}m E_{m}^{f}(0)\frac{1}{\lambda^m},
\end{aligned}
\end{equation}
which are the desired results. 
  \end{proof}
 The above result implies the following corollary.
 
 \begin{corollary}\label{Theorem 3.7}
Suppose $f$ satisfies the condition (\ref{(17)}). Then for $\lambda/u\in \mathbb{C}_{p},$  $|\lambda/u|_{p} >1$ and $|\lambda|_{p}>M,$  
there are identities 
\begin{equation}\label{(P3)}
 \zeta_{p}^{f}(s,\lambda+u)=\frac{1}{s-1}\langle \lambda\rangle^{1-s}\sum_{n=0}^{\infty}\binom{1-s}nB_{n}^{f}(u)\frac{1}{\lambda^{n}}
\end{equation}
and 
\begin{equation}\label{(P4)}
\zeta_{p,E}^{f}(s,\lambda+u)=\langle \lambda\rangle^{1-s}\sum_{n=0}^{\infty}\binom{1-s}nE_{n}^{f}(u)\frac{1}{\lambda^{n}}.
\end{equation}
If in addition  $\lambda$ and $u$ are so chosen such that $\lambda,$ $\lambda+u$ and $\lambda+u+f(a)~(a\in\mathbb{Z}_{p})$  lie in a
finite extension $K$ of $\mathbb{Q}_{p}$ whose ramification index
over $\mathbb{Q}_{p}$ is less than $p-1$, and let $(\pi)$ be the maximal ideal of
 the ring of integers $O_{K}$ of $K$, then (\ref{(P3)}) is valid for $s$ in
$\mathbb{C}_{p}$ such that $|s|_{p} <
 |\pi|_{p}^{-1}p^{-1/(p-1)}$ except for $s=1$, while (\ref{(P4)}) is valid for $s$ in
$\mathbb{C}_{p}$ such that $|s|_{p} <
 |\pi|_{p}^{-1}p^{-1/(p-1)}$.
\end{corollary} 
\begin{remark}
In \cite{KW}, Kimoto and Wakayama showed a corresponding formal power series expansion for the classical Hurwitz zeta function $\zeta(s,\lambda+u)$:
\begin{equation} \zeta(s,\lambda+u)=\lambda^{1-s}\sum_{n=0}^{\infty} (-1)^{n}B_{n}(u) \frac{\Gamma(n+s-1)}{\Gamma(n+1)\Gamma(s)}\frac{1}{\lambda^{n}}
\end{equation}
(see \cite[Example 6]{KW}) and  a convergent power series expansion for  the $p$-adic  Hurwitz zeta function $\zeta_{p}(s,\lambda+u)$:
\begin{equation} 
 \zeta_{p}(s,\lambda+u)=\frac{1}{s-1}\langle \lambda\rangle^{1-s}\sum_{n=0}^{\infty}\binom{1-s}nB_{n}(u)\frac{1}{\lambda^{n}}
\end{equation}
(see \cite[(50)]{KW}).
\end{remark}
\begin{proof}[Proof of Corollary \ref{Theorem 3.7}]
 Since by our assumption $|\lambda/u|_{p} >1$ and $|\lambda|_{p}>M$, we have
 $$|\lambda+u|_{p}=|\lambda|_{p}\left|1+\frac{u}{\lambda}\right|_{p}>M.$$
Then by (\ref{(P1)}) and (\ref{(P2)}) we have the following identities of analytic functions
 \begin{equation}\label{(3.11)}
\zeta_{p}^{f}(s,\lambda+u)=\frac{1}{s-1}\langle \lambda+u \rangle^{1-s}\sum_{m=0}^\infty\binom{1-s}m B_{m}^{f}(0)\frac{1}{(\lambda+u)^m}
\end{equation}
and 
\begin{equation}\label{(3.12)}
\zeta_{p,E}^{f}(s,\lambda+u)=\langle \lambda+u \rangle^{1-s}\sum_{m=0}^\infty\binom{1-s}m E_{m}^{f}(0)\frac{1}{(\lambda+u)^m}.
\end{equation} 
Since $|\lambda/u|_{p} >1$, we can write $\langle \lambda+ u\rangle=\langle \lambda\rangle \left(1+{u}/{\lambda}\right).$ Then by (\ref{(3.11)}) we get
 \begin{equation}\label{(3.13)}
 \begin{aligned}
\zeta_{p}^{f}(s,\lambda+u)&=\frac{1}{s-1}\langle \lambda+u \rangle^{1-s}\sum_{m=0}^\infty\binom{1-s}m B_{m}^{f}(0)\frac{1}{(\lambda+u)^m}\\
&=\frac{1}{s-1}\langle \lambda\rangle^{1-s}\sum_{m=0}^{\infty}\binom{1-s}m B_{m}^{f}(0)\frac{1}{\lambda^{m}}\left(1+\frac{u}{\lambda}\right)^{1-s-m}\\
&=\frac{1}{s-1}\langle \lambda\rangle^{1-s}\sum_{m=0}^{\infty}\binom{1-s}m B_{m}^{f}(0)\frac{1}{\lambda^{m}}\sum_{l=0}^{\infty}\binom{1-s-m}l u^{l}\frac{1}{\lambda^{l}} \\
&=\frac{1}{s-1}\langle \lambda\rangle^{1-s}\sum_{n=0}^\infty\sum_{m=0}^{n}\binom{1-s}m\binom{1-s-m}{n-m} B_{m}^{f}(0)u^{n-m}\frac{1}{\lambda^{n}}\\
&=\frac{1}{s-1}\langle \lambda\rangle^{1-s}\sum_{n=0}^{\infty}\binom{1-s}n\frac{1}{\lambda^{n}}\sum_{m=0}^{n}\binom{n}{m}B_{m}^{f}(0)u^{n-m}\\
&=\frac{1}{s-1}\langle \lambda\rangle^{1-s}\sum_{n=0}^{\infty}\binom{1-s}nB_{n}^{f}(u)\frac{1}{\lambda^{n}},\end{aligned}
\end{equation}
the fifth identity follows from the combinational identity $\binom{1-s}m\binom{1-s-m}{n-m}=\binom{1-s}n\binom{n}{m}$
and the last identity follows from (\ref{(B9+)}).
 Similarly, we have 
  $$\zeta_{p,E}^{f}(s,\lambda+u)=\langle \lambda\rangle^{1-s}\sum_{n=0}^{\infty}\binom{1-s}nE_{n}^{f}(u)\frac{1}{\lambda^{n}},$$
  which are the desired results. 
  \end{proof}
 The following derivative formulas for $\zeta_{p}^{f}(s,\lambda)$ and $\zeta_{p,E}^{f}(s,\lambda)$ are the consequence of the corresponding
 formula for the projection function  $\langle\lambda\rangle^{1-s}$ with respect to $\lambda$ (see \cite[p. 281, Lemma 11.2.3]{Cohen}).
 \begin{theorem}[Derivative formulas]\label{Theorem 3.10}
Suppose $f$ satisfies the condition (\ref{(17)}). Then for any $\lambda\in \mathbb{C}_{p}$ such that  $|\lambda|_{p} > M$ and for any $n\in\mathbb{N}$, we have
   \begin{equation}\label{(18)}
 \frac{\partial^n}{\partial\lambda^n}\zeta_{p}^{f}(s,\lambda)=\frac{(-1)^{n}}{\omega_{v}^{n}(\lambda)}(s-1)_{n}\zeta_{p}^{f}(s+n,\lambda)
 \end{equation}    
 and 
 \begin{equation}\label{(19)}
 \frac{\partial^n}{\partial\lambda^n}\zeta_{p,E}^{f}(s,\lambda)=\frac{(-1)^{n}}{\omega_{v}^{n}(\lambda)}(s-1)_{n}\zeta_{p,E}^{f}(s+n,\lambda),
 \end{equation}    
where $(a)_{n}=a(a+1)\cdots(a+n-1)=\frac{\Gamma(a+n)}{\Gamma(a)}$ is the Pochhammer symbol.
 \end{theorem}
 \begin{remark}
 In \cite{KW}, Kimoto and Wakayama showed the following derivative formula for the Hurwitz-type spectral zeta function of the quantum Rabi model (QRM)
 \begin{equation}\label{(19)}
 \frac{\partial^n}{\partial\lambda^n}\zeta_{\textrm{QRM}}(s,\lambda)=(-1)^{n}(s)_{n}\zeta_{\textrm{QRM}}(s+n,\lambda)
 \end{equation} 
 (see \cite[Lemma 4.6]{KW}).
   \end{remark}
\begin{proof}[Proof of Theorem \ref{Theorem 3.10}]
By \cite[p. 281, Lemma 11.2.3]{Cohen}, we have
$$\frac{\partial}{\partial\lambda} \langle\lambda+f(a)\rangle^{1-s}=(1-s)\frac{\langle \lambda+f(a) \rangle^{-s}}{\omega_{v}(\lambda+f(a))}$$
uniformly for $a\in\mathbb{Z}_{p}$.
Since $|\lambda|_{p}>M$ and $|f(a)|_{p}\leq M$ for $a\in\mathbb{Z}_{p}$, we get $|{f(a)}/{\lambda}|_{p}<1$ and $$\omega_v\left(1+\frac{f(a)}{\lambda}\right)=1.$$
Then from the multiplicity of the $p$-adic Teichm\"uller character $\omega_v(\lambda)$, we have
\begin{equation}\label{(2.31)}
\begin{aligned} 
\omega_v(\lambda+f(a))&=\omega_v(\lambda)\omega_v\left(1+\frac{f(a)}{\lambda}\right)\\
&=\omega_v(\lambda).
\end{aligned}
\end{equation}
Thus by (\ref{(Def1)}) and (\ref{(Def2)}), we get
 \begin{equation}
 \begin{aligned} 
 \frac{\partial}{\partial\lambda}\zeta_{p}^{f}(s,\lambda)&=\frac{1}{s-1} \frac{\partial}{\partial\lambda}\int_{\mathbb{Z}_{p}} \langle\lambda+f(a)\rangle^{1-s} da\\
&=\frac{1}{s-1} \int_{\mathbb{Z}_{p}}\frac{\partial}{\partial\lambda} \langle\lambda+f(a)\rangle^{1-s} da\\
&\quad\text{(see \cite[p. 171, Exercise 55.G.]{SC})}\\
&=\frac{1}{s-1}\frac{1-s}{\omega_{v}(\lambda+f(a))} \int_{\mathbb{Z}_{p}} \langle\lambda+f(a)\rangle^{-s} da\\
&=\frac{1-s}{\omega_{v}(\lambda)}\zeta_{p}^{f}(s+1,\lambda) 
 \end{aligned}
 \end{equation}
 and  
\begin{equation}
 \begin{aligned} 
 \frac{\partial}{\partial\lambda}\zeta_{p,E}^{f}(s,\lambda)&=\frac{\partial}{\partial\lambda}\int_{\mathbb{Z}_{p}} \langle\lambda+f(a)\rangle^{1-s} d\mu_{-1}(a)\\
&=\int_{\mathbb{Z}_{p}}\frac{\partial}{\partial\lambda} \langle\lambda+f(a)\rangle^{1-s} d\mu_{-1}(a)\\
&=\frac{1-s}{\omega_{v}(\lambda+f(a))} \int_{\mathbb{Z}_{p}} \langle\lambda+f(a)\rangle^{-s} d\mu_{-1}(a)\\
&=\frac{1-s}{\omega_{v}(\lambda)}\zeta_{p,E}^{f}(s+1,\lambda).
 \end{aligned}
 \end{equation}
From which, we obtain our results by the induction on $n$.
  \end{proof}

\section{p-adic log Gamma functions with respect to $f$}\label{Section 4}
In analogy with the definition of the functional determinant in the complex case (see (\ref{(11)})), in the following, we define the corresponding $p$-adic log Gamma functions 
(the $p$-adic functional determinants) as the derivatives of the Hurwitz zeta functions $\zeta_{p}^{f}(s,\lambda)$ and $\zeta_{p,E}^{f}(s,\lambda)$
 at $s=0$.
\begin{definition}\label{Definition 4.1}
For $\lambda\in\mathbb{C}_{p}$ such that $-\lambda\not\in\mathscr{F}$, we define 
(cf. \cite[p. 330, Definition 11.5.1]{Cohen})
\begin{equation}\label{(Def3)}
\log\Gamma_{p}^{f}(\lambda)=\omega_{v}(\lambda)\frac{\partial}{\partial s}\zeta_{p}^{f}(s,\lambda)\bigg|_{s=0}
\end{equation}
and 
\begin{equation}\label{(Def4)}
\log\Gamma_{p,E}^{f}(\lambda)=\omega_{v}(\lambda)\frac{\partial}{\partial s}\zeta_{p,E}^{f}(s,\lambda)\bigg|_{s=0}.
\end{equation}
\end{definition}
Then we have their integral representations as follows.
\begin{theorem}[Integral representations]\label{Theorem 4.2}
Suppose $f$ satisfies the condition (\ref{(17)}). Then for any $\lambda\in\mathbb{C}_{p}$ and $|\lambda|_{p}>M$, we have
\begin{equation}\label{G1}
\log\Gamma_{p}^{f}(\lambda)=\int_{\mathbb{Z}_{p}}(\lambda+f(a))(\log_{p}(\lambda+f(a))-1)da
\end{equation}
and 
\begin{equation}\label{G2}
\log\Gamma_{p,E}^{f}(\lambda)=-\int_{\mathbb{Z}_{p}}(\lambda+f(a))\log_{p}(\lambda+f(a))d\mu_{-1}(a).
\end{equation}
\end{theorem}
\begin{proof}
By (\ref{(Def1)}) we have
 \begin{equation} \label{re-va-1}
(s-1)\zeta_{p}^{f}(s,\lambda)=\int_{\mathbb{Z}_{p}}\langle \lambda+f(a) \rangle^{1-s} da.
\end{equation}
Deviating both sides of the above equation with respect to $s$ and setting $s=0$, we get
$$\zeta_{p}^{f}(0,\lambda)-\frac{\partial}{\partial s}\zeta_{p}^{f}(s,\lambda)\bigg|_{s=0}=-\int_{\mathbb{Z}_{p}}\langle\lambda+f(a)\rangle \log_{p}(\lambda+f(a))da,$$
which is equivalent to
\begin{equation}\label{(GP1)}
\begin{aligned}
\frac{\partial}{\partial s}\zeta_{p}^{f}(s,\lambda)\bigg|_{s=0}
&=\int_{\mathbb{Z}_{p}}\langle\lambda+f(a)\rangle \log_{p}(\lambda+f(a))da+\zeta_{p}^{f}(0,\lambda)\\
&=\int_{\mathbb{Z}_{p}}\langle\lambda+f(a)\rangle \log_{p}(\lambda+f(a))da-\int_{\mathbb{Z}_{p}}\langle\lambda+f(a)\rangle da\\
&=\int_{\mathbb{Z}_{p}}\langle\lambda+f(a)\rangle(\log_{p}(\lambda+f(a))-1)da.
\end{aligned}
\end{equation}
From (\ref{(2.31)}), we see that
\begin{equation}\label{(GP2)}
\begin{aligned}
 \lambda+f(a)&=\omega_v(\lambda+f(a))\langle\lambda+f(a)\rangle\\
 &=\omega_v(\lambda)\langle\lambda+f(a)\rangle.
  \end{aligned}
 \end{equation}
 So by (\ref{(GP1)}) and (\ref{(Def3)}), we have
 \begin{equation}
 \begin{aligned}
\log\Gamma_{p}^{f}(\lambda)&=\omega_{v}(\lambda)\frac{\partial}{\partial s}\zeta_{p}^{f}(s,\lambda)\bigg|_{s=0}\\
&=\omega_{v}(\lambda)\int_{\mathbb{Z}_{p}}\langle\lambda+f(a)\rangle(\log_{p}(\lambda+f(a))-1)da\\
&=\int_{\mathbb{Z}_{p}}(\lambda+f(a))(\log_{p}(\lambda+f(a))-1)da,
\end{aligned}
 \end{equation}
 which is (\ref{G1}).
 For (\ref{G2}), by (\ref{(Def2)}) we have
 \begin{equation}
\zeta_{p,E}^{f}(s,\lambda)=\int_{\mathbb{Z}_{p}}\langle \lambda+f(a) \rangle^{1-s} d\mu_{-1}(a).
\end{equation}
Deviating both sides of the above equation with respect to $s$ and setting $s=0$, by (\ref{(GP2)}), we get
\begin{equation}
\begin{aligned}
\log\Gamma_{p,E}^{f}(\lambda)&=\omega_{v}(\lambda)\frac{\partial}{\partial s}\zeta_{p,E}^{f}(s,\lambda)\bigg|_{s=0}\\
&=-\omega_{v}(\lambda)\int_{\mathbb{Z}_{p}}\langle\lambda+f(a)\rangle \log_{p}(\lambda+f(a))d\mu_{-1}(a)\\
&=-\int_{\mathbb{Z}_{p}}(\lambda+f(a)) \log_{p}(\lambda+f(a))d\mu_{-1}(a),
\end{aligned}
\end{equation}
which is (\ref{G2}).
\end{proof}
From the above integral representations, we have the following Stirling's series expansions of $\log\Gamma_{p}^{f}(\lambda)$ and $\log\Gamma_{p,E}^{f}(\lambda)$.
\begin{theorem}[Stirling's series]\label{Theorem 4.3}
Suppose $f$ satisfies the condition (\ref{(17)}). Then for any $\lambda\in\mathbb{C}_{p}$ and $|\lambda|_{p}>M$, we have
\begin{equation}
\begin{aligned}
\log\Gamma_{p}^{f}(\lambda)&=B_{1}^{f}(0)+\sum_{n=1}^{\infty}\frac{(-1)^{n+1}}{n(n+1)}B_{n+1}^{f}(0)\frac{1}{\lambda^{n}}\\
&\quad+ B_{1}^{f}(\lambda)\log_{p}\lambda-B_{1}^{f}(\lambda)\end{aligned}
\end{equation} 
and
\begin{equation}
\log\Gamma_{p,E}^{f}(\lambda)=-E_{1}^{f}(0)-\sum_{n=1}^{\infty}\frac{(-1)^{n+1}}{n(n+1)}E_{n+1}^{f}(0)\frac{1}{\lambda^{n}}
-E_{1}^{f}(\lambda)\log_{p}\lambda, 
\end{equation}
where $B_{m}^{f}(\lambda)$ and $E_{m}^{f}(\lambda)$ are the $m$th  Bernoulli and Euler polynomials associated with $f$, respectively.
\end{theorem}
\begin{proof} By the power series expansion of $\log(1+T)$, we have
\begin{equation}\label{(GP3)}
(1+T)\log(1+T)=T+\sum_{n=1}^{\infty}(-1)^{n+1}\frac{T^{n+1}}{n(n+1)}
\end{equation}
for $|T|_{p}<1$. Since $|\lambda|_{p}>M$ and $|f(a)|_{p}\leq M$ for $a\in\mathbb{Z}_{p}$, by (\ref{(GP3)}) we have 
$|{f(a)}/{\lambda}|_{p}<1$ and
\begin{equation}
\begin{aligned}
(\lambda+&f(a))\log_{p}(\lambda+f(a))-(\lambda+f(a))\\
&=\lambda\left(1+\frac{f(a)}{\lambda}\right)\log_{p}\left(1+\frac{f(a)}{\lambda}\right)+(\lambda+f(a))\log_{p}\lambda-(\lambda+f(a))\\
&=f(a)+\lambda\sum_{n=1}^{\infty}\frac{(-1)^{n+1}}{n(n+1)}\left(\frac{f(a)}{\lambda}\right)^{n+1}+(\lambda+f(a))\log_{p}\lambda-(\lambda+f(a)).
\end{aligned}
\end{equation}
Then substituting to (\ref{G1}), we have
\begin{equation}
\begin{aligned}
\log\Gamma_{p}^{f}(\lambda)&=\int_{\mathbb{Z}_{p}}(\lambda+f(a))(\log_{p}(\lambda+f(a))-1)da\\
&=\int_{\mathbb{Z}_{p}}f(a)da+\lambda\sum_{n=1}^{\infty}\frac{(-1)^{n+1}}{n(n+1)}
\left(\int_{\mathbb{Z}_{p}}f^{n+1}(a)da\right)\frac{1}{\lambda^{n+1}}\\
&\quad+\log_{p}\lambda\int_{\mathbb{Z}_{p}}(\lambda+f(a))da-\int_{\mathbb{Z}_{p}}(\lambda+f(a))da\\
&=B_{1}^{f}(0)+\sum_{n=1}^{\infty}\frac{(-1)^{n+1}}{n(n+1)}B_{n+1}^{f}(0)\frac{1}{\lambda^{n}}\\
&\quad+ B_{1}^{f}(\lambda)\log_{p}\lambda-B_{1}^{f}(\lambda),\end{aligned}
\end{equation}
the last equation follows from (\ref{(B9)}), the definition of the Bernoulli polynomials associated with  $f$. Similarly, we have
\begin{equation}
\begin{aligned}
(\lambda+&f(a))\log_{p}(\lambda+f(a)) \\
&=\lambda\left(1+\frac{f(a)}{\lambda}\right)\log_{p}\left(1+\frac{f(a)}{\lambda}\right)+(\lambda+f(a))\log_{p}\lambda\\
&=f(a)+\lambda\sum_{n=1}^{\infty}\frac{(-1)^{n+1}}{n(n+1)}\left(\frac{f(a)}{\lambda}\right)^{n+1}+(\lambda+f(a))\log_{p}\lambda.
\end{aligned}
\end{equation}
Then substituting to (\ref{G2}), we have
\begin{equation}
\begin{aligned}
\log\Gamma_{p,E}^{f}(\lambda)&=-\int_{\mathbb{Z}_{p}}(\lambda+f(a))\log_{p}(\lambda+f(a))d\mu_{-1}(a)\\
&=-\int_{\mathbb{Z}_{p}}f(a)d\mu_{-1}(a)-\lambda\sum_{n=1}^{\infty}\frac{(-1)^{n+1}}{n(n+1)}
\left(\int_{\mathbb{Z}_{p}}f^{n+1}(a)d\mu_{-1}(a)\right)\frac{1}{\lambda^{n+1}}\\
&\quad-\log_{p}\lambda\int_{\mathbb{Z}_{p}}(\lambda+f(a))d\mu_{-1}(a)\\
&=-E_{1}^{f}(0)-\sum_{n=1}^{\infty}\frac{(-1)^{n+1}}{n(n+1)}E_{n+1}^{f}(0)\frac{1}{\lambda^{n}}
-E_{1}^{f}(\lambda)\log_{p}\lambda,\end{aligned}
\end{equation}
the last equation follows from (\ref{(E10)}), the definition of the Euler polynomials associated with $f$.
\end{proof}

\section{Analytic interpolations}\label{Section 5}
Recall that $$D_{1}:=\{a\in\mathbb{C}_{p}: |a|_{p}\leq 1\}$$ denotes the unit disk of the $p$-adic complex plane $\mathbb{C}_{p}.$
In p. 5 of this paper, we have required  a  $p$-adic function $f(a)$ such that $f(a)~~\textrm{or}~~g(a)=\frac{1}{f(a)}$ is locally analytic on $D_{1}$, which may interpolate the spectrum for the Hamiltonian of a quantum mode.
As suggested by the referee, in this section, we try to show two algorithms to determine the $p$-adic analytic function $f$ on $D_{1}$ which interpolates the spectra of a quantum model.

As stated above, by solving the Schr\"odinger equation (\ref{(1)}), we obtain the spectrum of the Hamiltonian $H$: $\{\lambda_n\}_{n=0}^\infty$. 
Let $K$ be a finite extension of $\mathbb{Q}_p$ embedded in $\mathbb{C}_p$. 
We assume all eigenvalues $\lambda_n$ ($n \in \mathbb{N}_0$) belong to $K$ after suitable normalization.

In the following, we will consider two methods which are investigated in Iwasawa's well-known book \cite{Iw}.
 Let 
$$
c_n = \sum_{i=0}^n (-1)^{n-i} \binom{n}{i} \lambda_i, \quad n\in\mathbb{N}_{0}
$$
(see \cite[p. 22]{Iw}). 
By \cite[p. 34]{Iw}, there exists a power series \(A(x)\) in \(K[[x]]\) which converges in a circle of radius \(>1\) around \(0\) and which satisfies
\[
A(n) = \lambda_n, \quad n\in\mathbb{N}_{0}
\]
if and only if
\[
\lim_{n \to \infty} |c_n| r^{-n} = 0,
\]
for a real number \( r \) such that
\[
0 < r < |p|^{\frac{1}{p-1}}.
\]
In this case, we have \[
A(\xi) = \lim_{k \to \infty} \sum_{i=0}^k c_i \binom{\xi}{i} = \sum_{n=0}^\infty c_n \binom{\xi}{n}
\]
for  each \(\xi\) in \(\mathbb{C}_p\) with \(|\xi| < |p|^{p^{-1}} r^{-1}\) (see \cite[p. 25, Corollary]{Iw}). Since \(A(x)\) in \(K[[x]]\) and it converges in a circle of radius \(>1\), we have $A(x)$ is (locally) analytic on $D_{1}$.

The second approach is due to Mahler. Since $\mathbb{Z}_{p}\subset D_{1}$ and $f$ is locally analytic on $D_{1}$, we have $f$ is  continuous  on $\mathbb{Z}_{p}$.
By a classical theorem of Mahler (see \cite[p. 35]{Iw}): there exists a continuous map
\[
f: \mathbb{Z}_p \to K
\]
satisfying
\[
f(n) = \lambda_n, \quad n\in\mathbb{N}_{0}
\]
if and only if
\[
\lim_{n \to \infty} |c_n| = 0.
\]
In addition, if this condition is satisfied, then \(f\) is unique and will be given by
\[
f(x) = \sum_{n=0}^{\infty} c_n \binom{x}{n}, \quad x \in \mathbb{Z}_{p}.
\]

\section*{Author Declarations}
\textbf{Conflict of Interest}  The authors have no conflicts to disclose.

\section*{Data Availability}
Data sharing is not applicable to this article as no new data were created or analyzed in this study.

\section*{Acknowledgements} 
 The authors are enormously grateful to the anonymous referee for his/her very careful
reading of this paper, and for his/her many valuable and detailed suggestions. 

Su Hu is supported by is supported by the Natural Science Foundation of Guangdong Province, China (No. 2024A1515012337).  Min-Soo Kim is supported by the National Research Foundation of Korea(NRF) grant funded by the Korea government(MSIT) (No. NRF-2022R1F1A1065551).

\bibliography{central}

\end{document}